\documentclass{amsart}
\usepackage{amsmath, amsfonts, amssymb, amsthm, amscd, enumerate}
\usepackage{hyperref}

\newtheorem{theorem}{Theorem}[section]
\newtheorem{lemma}[theorem]{Lemma}
\newtheorem{corollary}[theorem]{Corollary} 

\newtheorem{proposition}[theorem]{Proposition}

\newtheorem{intthm}{Theorem}

\theoremstyle{definition}
\newtheorem{definition}[theorem]{{Definition}}
\newtheorem{example}[theorem]{Example}
\newtheorem{remark}[theorem]{Remark}

\newtheorem{assumption}[theorem]{Assumption}

\newtheorem{para}[theorem]{}

\newtheorem*{chunk*}{}

\numberwithin{equation}{theorem}

\input xy
\xyoption{all}

\def\p{{\mathfrak p}}
\def\q{{\mathfrak q}}

\def\m{{\mathfrak m}}

\def\Hom{\operatorname{Hom}}

\def\ker{\operatorname{Ker}}
\def\im{\operatorname{Im}}

\newcommand{\Spec}{\operatorname{Spec}}
\def\Cl{\operatorname{Cl}}

\def\UFD{\operatorname{UFD}}

\def\p{\mathfrak{p}}

\newcommand{\ol}{\overline}
\newcommand{\Ht}{\operatorname{ht}}
\newcommand{\pd}{\operatorname{pd}}
\newcommand{\fp}{\mathfrak p}
\newcommand{\n}{\mathfrak n}
\newcommand{\fa}{\mathfrak a}

\newcommand{\id}{\operatorname{id}}
\newcommand{\dual}[2]{#1^{#2}}
\newcommand{\ddual}[2]{#1^{#2#2}}
\newcommand{\dddual}[2]{#1^{#2#2#2}}
\newcommand{\pdual}[2]{(#1)^{#2}}
\newcommand{\pddual}[2]{(#1)^{#2#2}}
\newcommand{\pdddual}[2]{(#1)^{#2#2#2}}

\newcommand{\bddual}[2]{[#1]^{#2#2}}
\newcommand{\bidual}[2]{\sigma^{#1}_{#2}}
\newcommand{\cl}[1]{[#1]}
\newcommand{\fb}{\mathfrak b}
\newcommand{\vf}{\varphi}
\renewcommand{\ker}{\operatorname{Ker}}
\newcommand{\xra}{\xrightarrow}

\begin{document}
\bibliographystyle{amsplain}



\title[Induced Maps on Divisor Class Groups]
{Maps on Divisor Class Groups Induced by Ring Homomorphisms of Finite
Flat Dimension}

\author{Sean Sather-Wagstaff}
\address{Sean Sather-Wagstaff, Department of Mathematics,
NDSU Dept \# 2750,
PO Box 6050,
Fargo, ND 58108-6050,
USA }
\email{Sean.Sather-Wagstaff@ndsu.edu}
\urladdr{http://math.ndsu.nodak.edu/faculty/ssatherw/}

\author{Sandra Spiroff}
\address{University of Mississippi
Department of Mathematics
305 Hume Hall
P.O. Box 1848
University , MS 38677-1848 }
\email{spiroff@olemiss.edu}
\urladdr{http://home.olemiss.edu/\~{}spiroff/}

\begin{abstract} 
Let $\vf\colon A\to B$ be a ring homomorphism between Noetherian normal integral domains.
We establish a general criterion for $\vf$ to induce a homomorphism 
$\Cl(\vf)\colon\Cl(A)\to\Cl(B)$
on divisor class groups. For instance, this criterion
applies whenever $\vf$ has finite flat dimension; this special case generalizes
the more classical situations where $\vf$ is flat or is surjective with kernel generated
by an $A$-regular element. We 
extend some of Spiroff's work on the kernels of 
induced maps to this more general  setting. 
\end{abstract}

\date \today
\maketitle

\section*{Introduction}

The divisor class group of a Noetherian normal integral domain $A$,
denoted $\Cl(A)$, measures certain aspects of the factorization-theory of $A$.
For instance, it is well-known that $A$ is a unique factorization domain
if and only if $\Cl(A)$ is trivial. For definitions and notation, consult the beginning
of Section~\ref{sec04}. 

It is natural to investigate the transfer of such factorization properties
between rings that are connected by a ring homomorphism. 
As part of such an investigation, one should find nontrivial classes of ring homomorphisms
$\vf\colon A\to B$ of Noetherian normal integral domains
that induce group homomorphisms $\Cl(\vf)\colon \Cl(A)\to \Cl(B)$.
For instance, the flat ring homomorphisms have this property.
Danilov~\cite[Prop.\ 1.1]{D1} shows that the natural surjection 
$B[\![T]\!]\to B$ also has this property, and Lipman~\cite[\S 0]{L} extended this
to any surjection of the form $A\to A/fA$, assuming that $A$ and $A/fA$
are both Noetherian normal integral domains.

In this paper, we establish a general criterion guaranteeing that a ring homomorphism
$\vf\colon A\to B$ of Noetherian normal integral domains induces a 
group homomorphism $\Cl(\vf)\colon \Cl(A)\to \Cl(B)$; see
Theorem~\ref{prop0101}. 
This was established, with extra assumptions, by Spiroff in her dissertation~\cite{spiroff:diss},
but the details cannot be found in
any of the current literature.   
We include a complete proof here because this criterion has been used without proof
by several authors, including Griffith and Weston~\cite{GW},
Sather-Wagstaff~\cite{SSW},
and Spiroff~\cite{Sp}, and because we require the general statement  
for our own investigation.
As a special case of this we have the following result,
proved in~\eqref{proofA},
which shows that our criterion encompasses the cases described in the previous paragraph:

\begin{intthm}  \label{prop0102}
Let $\vf\colon A \to B$ be a ring homomorphism of finite flat dimension
between Noetherian normal integral domains.
Then there is a well-defined group homomorphism
$\Cl(\vf)\colon\Cl(A)\to\Cl(B)$ given by
$\cl\fa\mapsto \cl{\pddual{\fa\otimes_A B}{B}}$,
where $\pddual{-}{B}$ is the double-dual with respect to $B$.
\end{intthm}

The injectivity of the map $\Cl(\vf)$ has been studied by several authors, including
Danilov~\cite{D3, D2, D1}, Griffith and Weston~\cite{GW}, and Spiroff~\cite{Sp}. 
The work in~\cite{Sp} is guided by
the following principle: When $A$ is a local Noetherian normal 
integral domain, the pathological behavior exhibited by maps of the form 
$\Cl(\vf)\colon\Cl(A)\to\Cl(A/I_n)$
lies near the ``top'' of the maximal ideal.
In the current paper, we extend this idea to the more general setting
of induced maps guaranteed to exist by Theorem~\ref{prop0102}.
The outcome is the next result which is proved in~\eqref{proofB}.

\begin{intthm} \label{thm0201}
Let $(A, \m, k$) be an excellent normal local integral
domain and let $I_1, I_2, \dots$ be a sequence of 
nonzero prime ideals in $A$ of finite projective dimension
such that the $\lim_{n \to \infty} I_n = 0$ in the $\m$-adic
topology and such that $A/I_n$ is a normal integral domain for each $n$.
If $\pi_n\colon A\to A/I_n$ is the natural surjection, 
then $\bigcap_{n=1}^\infty \ker(\Cl(\pi_n))$ is trivial. 
\end{intthm}

Here we summarize the contents of this paper. Section~\ref{sec04}
begins with requisite background information and contains the proof of our
criterion for the existence of induced homomorphisms of divisor class groups. 
It concludes with a discussion of the functorial behavior of our construction.
Section~\ref{sec02} contains an investigation of the injectivity of our induced maps.
We end the paper with Section~\ref{sec03}, which contains examples
of ideals satisfying the hypotheses of Theorem~\ref{thm0201}.

\textbf{Assumptions.}
In this paper, the term ``normal integral domain'' is short
for ``Noetherian integral domain that is integrally closed in its
field of fractions''. We assume without reference 
results from the texts of Matsumura \cite{M} and Nagata \cite{N}.

\section{Divisor Class Groups and Induced Homomorphisms} \label{sec04}

In this section, we show how certain homomorphisms of normal integral 
domains induce
homomorphisms on the corresponding divisor class groups.
We begin with 
our working definition of the divisor class group of a  normal integral domain.
It  can be found in Lipman~\cite[\S 0]{L} and is 
equivalent to the classical additive definition of the divisor class 
group appearing in Bourbaki~\cite[Ch.\ VII]{B} and Fossum~\cite[\S 6]{F};
see, e.g., Sather-Wagstaff~\cite[2.10]{SSW} for a discussion of this equivalence.  

\begin{definition} \label{defn0101}
Let $A$ be a  normal integral domain and $M$ a
finitely generated $A$-module .  The \emph{dual} of $M$ is 
$\dual{M}{A}=\Hom_A(M, A)$ and the \emph{double dual} of $M$ is 
$\ddual{M}{A} = \pdual{\dual{M}{A}}{A}$.  The natural \emph{biduality map}
$\bidual{A}{M}\colon M \to \ddual{M}{A}$ is the $A$-module homomorphism 
given by $\bidual{A}{M}(m)(g)=g(m)$  for all $m \in M$ and all $g \in \dual{M}{A}$.  
We say that $M$ is \emph{reflexive} if $\bidual{A}{M}$ is an isomorphism.  
\end{definition}

\begin{remark} \label{rmk0101}
In much of the literature, the dual of an $A$-module $M$ is denoted $M^{\ast}$.
We have chosen the notation $\dual{M}{A}$ in order to avoid ambiguity when we work with
two or more rings simultaneously, as in the proof of Lemma~\ref{lem0104}.
\end{remark}
 
\begin{definition} \label{defn0102}
Let $A$ be a  normal integral domain.
The \emph{divisor class group} of $A$, denoted $\Cl(A)$, is the group 
of isomorphism classes of reflexive $A$-modules of rank one, or equivalently, reflexive ideals of $A$.  
An element $\cl\fa \in \Cl(A)$ is called a \emph{divisor class}, and multiplication is
defined by $\cl\fa \cdot \cl\fb = \cl{\pddual{\fa \otimes_A \fb}{A}}$.
The identity element is $\cl A$, and 
$\cl{\fa}^{-1}=\cl{\dual{\fa}{A}}$.
\end{definition}

\begin{remark} \label{functor} 
Let $A$ be a  normal integral domain.
For each  homomorphism of finitely generated $A$-modules $\psi\colon M\to N$,
the functoriality of $\Hom_A(-,A)$ yields $A$-module homomorphisms 
$\dual{\psi}{A}\colon \dual{N}{A}\to \dual{M}{A}$ and
$\ddual{\psi}{A}\colon \ddual{M}{A}\to \ddual{N}{A}$.
Also, the naturality of the biduality map manifests itself in
the next commutative diagram.
$$\xymatrix{
M \ar[r]^{\psi} \ar[d]_{\bidual{A}{M}}
& N \ar[d]^{\bidual{A}{N}} \\
\ddual{M}{A} \ar[r]^{\ddual{\psi}{A}}
&\ddual{N}{A}
}$$
\end{remark}

We  use the following result of Auslander and Buchsbaum in several key places.

\begin{lemma}[\protect{\cite[Prop.\ 3.4]{AB}}] \label{AG} Let $A$ be 
a  normal integral domain.  If $M$ 
is a reflexive $A$-module, $N$ is a torsion-free $A$-module, and 
$\psi\colon M \to N$ is an $A$-module homomorphism, then $\psi$ is an isomorphism if
and only if, for each prime ideal $\p\in\Spec(A)$ of height at most one, the induced homomorphism 
$\psi_{\p}\colon M_{\p}\to N_{\p}$
is an isomorphism.
\end{lemma}

The next three lemmas are for use in the proof of Theorem~\ref{prop0101}.

\begin{lemma}  \label{lem0102}
Let $\vf\colon A \to B$ be a ring homomorphism of integral domains 
and set $\q=\ker(\vf)$. 
If $M$ is a finitely generated $A$-module such that $M_{\q}\cong A_{\q}^r$, then
$M\otimes_AB$ is a finitely generated $B$-module of rank $r$.
\end{lemma}

\begin{proof}
We need to show that 
$(M \otimes_A B)\otimes_BL\cong L^r$ where $L$ is the field of fractions of $B$.
Using the natural commutative diagram of ring homomorphisms
$$
\xymatrix{
A \ar@{->>}[r] \ar@{^(->}[d] 
& A/\q \ar@{^(->}[r]^-{\vf'}\ar@{^(->}[d]
& B \ar@{^(->}[d] \\
A_{\q}\ar@{->>}[r]
& (A/\q)_{\q} \ar@{^(->}[r]^-{(\vf')_{\q}}
& B_{\q} \ar@{^(->}[r]
&L
}$$
we have the following isomorphisms
$$(M\otimes_AB)\otimes_BL
\cong (M\otimes_AA_{\q})\otimes_{A_{\q}}L
\cong M_{\q}\otimes_{A_{\q}}L
\cong A_{\q}^r\otimes_{A_{\q}}L
\cong L^r
$$
which yield the desired conclusion.
\end{proof}

\begin{lemma} \label{lem0103}
Let $B$ be a normal integral domain, and let $M$ and $N$ be finitely
generated $B$-modules. Assume that, for each prime ideal $P\in\Spec(B)$
of height at most one, the localizations $M_P$ and $N_P$ are reflexive $B_P$-modules.
Then there is an isomorphism
$\pddual{M\otimes_BN}{B}\cong
\pddual{\ddual{M}{B}\otimes_B\ddual{N}{B}}{B}$.
\end{lemma}

\begin{proof}
Fix a prime ideal $P\in\Spec(B)$
of height at most one. 
By assumption $M_P$ is a reflexive $B_P$-module. Using the natural commutative
diagram 
$$\xymatrix{
M_P\ar[r]^-{\bidual{B_P}{M_P}}_-{\cong} \ar[rd]_-{(\bidual{B}{M})_P}
& \pddual{M_P}{B_P} \ar[d]^{\cong} \\
&(\ddual{M}{B})_P
}$$
it follows that the map $(\bidual{B}{M})_P$ is an isomorphism.
Similarly, we deduce that the map $(\bidual{B}{N})_P$ is an isomorphism.
From this we conclude that the tensor product
$$(\bidual{B}{M})_P\otimes_{B_P}(\bidual{B}{N})_P\colon
M_P\otimes_{B_P}N_P\to\ddual{M_P}{B_P}\otimes_{B_P}\ddual{N_P}{B_P}$$
is an isomorphism.
Using the commutative diagram
$$\xymatrix{
M_P\otimes_{B_P}N_P
\ar[rrr]^-{(\bidual{B}{M})_P\otimes_{B_P}(\bidual{B}{N})_P}_-{\cong} 
\ar[d]_{\cong}
&&&\ddual{M_P}{B_P}\otimes_{B_P}\ddual{N_P}{B_P} \ar[d]^{\cong} \\
(M\otimes_BN)_P
\ar[rrr]^{(\bidual{B}{M}\otimes_B\bidual{B}{N})_P}
&&& (\ddual{M}{B}\otimes_B\ddual{N}{B})_P
}$$
it follows that the map $(\bidual{B}{M}\otimes_B\bidual{B}{N})_P$
is an isomorphism.
From this, we conclude that the top row of the next commutative diagram
is an isomorphism
$$\xymatrix{
\bddual{(M\otimes_BN)_P}{B_P}
\ar[rrr]^-{\bddual{(\bidual{B}{M}\otimes_B\bidual{B}{N})_P}{B_P}}_-{\cong}
\ar[d]_{\cong}
&&& \bddual{(\ddual{M}{B}\otimes_B\ddual{N}{B})_P}{B_P}\ar[d]^{\cong}\\
[\pddual{M\otimes_BN}{B}]_P
\ar[rrr]^-{[\pddual{\bidual{B}{M}\otimes_B\bidual{B}{N}}{B}]_P}
&&&[\pddual{\ddual{M}{B}\otimes_B\ddual{N}{B}}{B}]_P
}$$
and so the map $[\pddual{\bidual{B}{M}\otimes_B\bidual{B}{N}}{B}]_P$
is an isomorphism. 

Since the ring $B$ is a normal integral domain and the modules 
$\pddual{\ddual{M}{B}\otimes_B\ddual{N}{B}}{B}$
and
$\pddual{M\otimes_BN}{B}$ 
are each of the form $\Hom_B(U,B)$ for some finitely generated $B$-module $U$,
each of these $B$-modules is reflexive.
Since the homomorphism 
$$\pddual{\bidual{B}{M}\otimes_B\bidual{B}{N}}{B}\colon
\pddual{M\otimes_BN}{B}\to \pddual{\ddual{M}{B}\otimes_B\ddual{N}{B}}{B}
$$
localizes to an isomorphism at each prime ideal of height at most one,
Lemma~\ref{AG} implies that it is an isomorphism,
yielding the desired conclusion.
\end{proof}

\begin{lemma} \label{lem0104}
Let $\vf\colon A\to B$ be a ring homomorphism of integral domains
such that $B$ is normal.
Let $M$ be a finitely
generated $A$-module such that, for each prime ideal $P\in\Spec(B)$ of height at most one,
the $A_{\p}$-module $M_{\p}$ is reflexive, where $\p=\vf^{-1}(P)$.
Then there is an isomorphism
$\pddual{M\otimes_AB}{B}\cong
\pddual{\ddual{M}{A}\otimes_AB}{B}$.
\end{lemma}

\begin{proof}
As in the proof of Lemma~\ref{lem0103}, it suffices to construct a $B$-module
homomorphism $\psi\colon \pddual{M\otimes_AB}{B}\to
\pddual{\ddual{M}{A}\otimes_AB}{B}$ such that, for each
prime ideal $P\in\Spec(B)$ of height at most one, the localization
$\psi_P$ is an isomorphism. We show that the map
$\psi=\pddual{\bidual{A}{M}\otimes_AB}{B}$ satisfies the desired property.

Fix a prime ideal $P\in\Spec(B)$ of height at most one, and set  $\p=\vf^{-1}(P)$.
The $A_{\p}$-module $M_{\p}$ is  reflexive by assumption,
so the biduality map $\bidual{A_{\p}}{M_{\p}}\colon M_{\p}\to
\pddual{M_{\p}}{A_{\p}}$ is an isomorphism.
It follows that the induced map
$$\pddual{\bidual{A_{\p}}{M_{\p}}\otimes_{A_{\p}}B_P}{B_P}
\colon
\pddual{M_{\p}\otimes_{A_{\p}}B_P}{B_P}
\to
\pddual{\pddual{M_{\p}}{A_{\p}}\otimes_{A_{\p}}B_P}{B_P}$$
is an isomorphism as well.
From the natural commutative diagram
$$\xymatrix{
\pddual{M_{\p}\otimes_{A_{\p}}B_P}{B_P}
\ar[rrr]^-{\pddual{\bidual{A_{\p}}{M_{\p}}\otimes_{A_{\p}}B_P}{B_P}}_{\cong}
\ar[d]_{\cong}
&&& \pddual{\pddual{M_{\p}}{A_{\p}}\otimes_{A_{\p}}B_P}{B_P} 
\ar[d]^{\cong} \\
[\pddual{M\otimes_AB}{B}]_P
\ar[rrr]^-{[\pddual{\bidual{A}{M}\otimes_AB}{B}]_P}
&&&[\pddual{\ddual{M}{A}\otimes_AB}{B}]_P
}$$
it follows that $[\pddual{\bidual{A}{M}\otimes_AB}{B}]_P$ is an isomorphism
as desired.
\end{proof}

The next definition is our final preparation for
Theorem~\ref{prop0101}.

\begin{definition} \label{defn0112}
For each normal integral domain $A$, we set
$$\UFD(A) = \{\fp \in \Spec(A) \mid \text{$A_{\fp}$ is a unique factorization domain}\}.$$
\end{definition}

We now state and prove  the main result of this section.

\begin{theorem}  \label{prop0101}
Let $\vf\colon A \to B$ be a ring homomorphism of  normal integral domains 
such that, for each prime ideal $P\in\Spec(B)$ of height at most one,
we have $\vf^{-1}(P)\in\UFD(A)$.
Then there is a well-defined group homomorphism
$\Cl(\vf)\colon\Cl(A)\to\Cl(B)$ given by
$\cl\fa\mapsto \cl{\pddual{\fa\otimes_A B}{B}}$.
\end{theorem}

\begin{proof}
Let $K$ and $L$ denote the fields of fractions of $A$ and $B$, respectively.  
Set $\q = \ker(\vf$), which is a prime ideal of $A$.   


To show that $\Cl(\vf)$ is well-defined,
it suffices to fix
a reflexive $A$-module $\fa$ of rank one and verify
that $\pddual{\fa \otimes_A B}{B}$
is a reflexive $B$-module of rank one.  
The reflexivity follows from the fact that $B$ is normal and $\pddual{\fa \otimes_A B}{B}$
is of the form $\Hom_B(U,B)$ for some  finitely generated $B$-module $U$.
Lemma~\ref{lem0102} shows that the $B$-module $\fa\otimes_AB$ has rank one,
that is, there is an isomorphism $(\fa \otimes_A B)\otimes_BL\cong L$.
This yields the second isomorphism in the next sequence
$$\pddual{\fa \otimes_A B}{B}\otimes_BL\cong
\pddual{(\fa \otimes_A B)\otimes_BL}{L}\cong
\pddual{L}{L}\cong L.$$
The first isomorphism follows from the fact that $L$ is flat as a $B$-module,
since $\fa\otimes_AB$ is  finitely generated.
The third isomorphism
is routine. 
This shows that $\pddual{\fa \otimes_A B}{B}$
has rank one, 
and it follows that $\Cl(\vf)$ is well-defined.

To complete the proof, we  show that $\Cl(\vf)$ respects the group 
structures of $\Cl(A)$ and $\Cl(B)$.  
To this end, fix two rank-one reflexive $A$-modules $\fa$ and $\fb$.
The cancellation isomorphism 
$(\fa\otimes_AB)\otimes_B(\fb\otimes_AB)\cong(\fa\otimes_A\fb)\otimes_AB$
yields the unlabeled isomorphism in the following sequence: 
\begin{align*}
\bddual{\pddual{\fa\otimes_A\fb}{A}\otimes_A B}{B}
&\stackrel{(1)}{\cong} \bddual{(\fa\otimes_A\fb)\otimes_A B}{B} \\
&\cong \bddual{(\fa\otimes_AB)\otimes_B(\fb\otimes_AB)}{B} \\
&\stackrel{(2)}{\cong}\bddual{\pddual{\fa\otimes_AB}{B}\otimes_B\pddual{\fb\otimes_AB}{B}}{B}.
\end{align*}

To justify isomorphism (1), we  show that the hypotheses of
Lemma~\ref{lem0104} are satisfied with the module $M=\fa\otimes_A\fb$.
Let $P\in\Spec(B)$ be a prime ideal of height at most one, and set $\p=\vf^{-1}(P)$.
By assumption, the group $\Cl(A_{\p})$ is trivial.
The fact that $\fa_{\p}$ and $\fb_{\p}$ are reflexive $A_{\p}$-modules of
rank one implies $\fa_{\p}\cong A_{\p}\cong\fb_{\p}$ and so
$$(\fa\otimes_A\fb)_{\p}\cong\fa_{\p}\otimes_{A_{\p}}\fb_{\p}
\cong A_{\p}\otimes_{A_{\p}} A_{\p}\cong A_{\p}.$$
In particular, this is a reflexive $A_{\p}$-module, so Lemma~\ref{lem0104}
yields isomorphism (1).

For isomorphism (2), we  show that the hypotheses of
Lemma~\ref{lem0103} are satisfied with the module $M=\fa\otimes_AB$ and $N=\fb\otimes_AB$.
It suffices (by symmetry) to show that, for each  prime ideal $P\in\Spec(B)$ of height at most one, 
the $B_P$-module $(\fa\otimes_AB)_P$ is reflexive.  
Set $\p=\vf^{-1}(P)$.
In the following sequence of isomorphisms
$$(\fa\otimes_AB)_P
\cong (\fa\otimes_AB)\otimes_BB_P
\cong (\fa\otimes_AA_{\p})\otimes_{A_{\p}}B_P
\cong \fa_{\p}\otimes_{A_{\p}}B_P
\cong A_{\p}\otimes_{A_{\p}}B_P
\cong B_P$$
the second isomorphism comes from the following commutative diagram of ring homomorphisms
$$\xymatrix{
A \ar[r]^-{\vf} \ar[d] & B \ar[d] \\
A_{\p} \ar[r]^-{\vf_P} & B_P}$$
and the remaining isomorphisms are straightforward.
The $B_P$-module $(\fa\otimes_AB)_P\cong B_P$ is reflexive,
thus completing the proof.
\end{proof}

We next show how Theorem~\ref{prop0102} follows from Theorem~\ref{prop0101}.

\begin{para} \label{proofA}
\textbf{Proof of Theorem~\ref{prop0102}.}
Let $P\in\Spec(B)$ be a prime ideal of height at most one, and set $\fp=\vf^{-1}(P)$.
The localized homomorphism $\vf_P\colon A_{\fp}\to B_P$ has finite flat 
dimension.  Since $B$ is normal and $\operatorname{ht}(P)\leq 1$,
the ring $B_P$ is regular.  It follows from~\cite[Thm.\ 6.1.1]{AFHa} that
$A_{\fp}$ is regular, and hence $\fp\in\UFD(A)$.
Hence, the desired result follows from Theorem~\ref{prop0101}.
\qed
\end{para}

\begin{remark} \label{rmk0102}
Theorem~\ref{prop0102} unifies two classical situations
of induced maps on divisor class groups. The first situation is
when $\vf$ is flat, and the second situation is when $\vf$ is surjective
with kernel generated by an $A$-regular element; see~\cite[\S 0]{L}. Indeed, in each
of these situations, the map $\vf$ has finite flat dimension. 
Another classical case is covered by Theorem~\ref{prop0101}, namely,
the case where $\vf$ is injective and integral. In this case, the Cohen-Seidenberg Theorems imply that
$\Ht(P)=\Ht(\vf^{-1}(P))$ for every prime ideal of $B$. In particular,
when $\Ht(P)\leq1$ we have $\Ht(\vf^{-1}(P))\leq1$; this implies that $A_{\vf^{-1}(P)}$
is regular and hence $\vf^{-1}(P)\in\UFD(A)$.
Corollary~\ref{cor0101} shows how
to combine these classical situations.
\end{remark}

\begin{remark} \label{rmk0103}
Continue with the notation of Theorem~\ref{prop0101}.
For each divisor class $\cl\fa\in\Cl(A)$, we have
$$\Cl(\vf)(\cl\fa^{-1})
=\Cl(\vf)(\cl\fa)^{-1}
=\cl{\pddual{\fa\otimes_AB}{B}}^{-1}
=\cl{\pdddual{\fa\otimes_AB}{B}}
=\cl{\pdual{\fa\otimes_AB}{B}}.$$
The first equality follows from the fact that $\Cl(\vf)$ is a
homomorphism of multiplicative groups. The second and third equalities
are by definition. 
To verify the fourth equality, it suffices to show that
the biduality morphism
$$\bidual{B}{\pdual{\fa\otimes_AB}{B}}\colon
\pdual{\fa\otimes_AB}{B}\to\pdddual{\fa\otimes_AB}{B}$$
is an isomorphism. As in the proof of Lemma~\ref{lem0104},
it is straightforward to show that, for each prime ideal $P\in\Spec(B)$
of height at most one,
the induced map 
$$(\bidual{B}{\pdual{\fa\otimes_AB}{B}})_{P}\colon
(\pdual{\fa\otimes_AB}{B})_P\to(\pdddual{\fa\otimes_AB}{B})_P$$
is an isomorphism. 
The $B$-modules $\pdual{\fa\otimes_AB}{B}$ and $\pdddual{\fa\otimes_AB}{B}$
are both reflexive because
$B$ is a normal integral domain and each of the modules
is of the form $\Hom_B(U,B)$ for some finitely generated
$B$-module $U$.
Hence, Lemma~\ref{AG} implies that $\bidual{B}{\pdual{\fa\otimes_AB}{B}}$
is an isomorphism.
\end{remark}

In the final results of this section,  we analyze the functoriality of the operator $\Cl(-)$.

\begin{theorem}  \label{thm0101}
Let $\vf\colon A \to B$ and $\psi\colon B\to C$ 
be  ring homomorphisms of  normal integral domains. 
If each of the maps $\vf$, $\psi$ and $\psi\circ\vf$ satisfies
the hypotheses of Theorem~\ref{prop0101},
then $\Cl(\psi\circ\vf)=\Cl(\psi)\circ\Cl(\vf)$.
\end{theorem}

\begin{proof}
\newcommand{\fP}{\mathfrak{P}}
Fix an element $\cl\fa\in\Cl(A)$. To show that
$\Cl(\psi\circ\vf)(\cl\fa)=\Cl(\psi)(\Cl(\vf)(\cl\fa))$,
we need to exhibit an isomorphism of $C$-modules
$$\pddual{\fa\otimes_AC}{C}
\cong\bddual{\pddual{\fa\otimes_AB}{B}\otimes_BC}{C}.$$
We start with the biduality map
$$\bidual{B}{\fa\otimes_AB}\colon\fa\otimes_AB\to\pddual{\fa\otimes_AB}{B}$$
and apply the functor $-\otimes_BC$ to yield the $C$-module homomorphism
$$\bidual{B}{\fa\otimes_AB}\otimes_BC\colon(\fa\otimes_AB)\otimes_BC
\to\pddual{\fa\otimes_AB}{B}\otimes_BC.$$
Let $\Upsilon\colon \fa\otimes_AC\to(\fa\otimes_AB)\otimes_BC$
denote the natural tensor-cancellation isomorphism and set
$$\delta=(\bidual{B}{\fa\otimes_AB}\otimes_BC)\circ \Upsilon\colon\fa\otimes_AC\to
\pddual{\fa\otimes_AB}{B}\otimes_BC.$$
We claim that the double-dual
$$\ddual{\delta}{C}\colon \pddual{\fa\otimes_AC}{C}
\to\bddual{\pddual{\fa\otimes_AB}{B}\otimes_BC}{C}$$
is an isomorphism. Lemma~\ref{AG} says that we need only show that the induced
map $(\ddual{\delta}{C})_{\fP}$ is an isomorphism for each prime ideal
$\fP\in\Spec(C)$ such that $\Ht(\fP)\leq1$. Fix such a prime $\fP$ and set
$P=\psi^{-1}(\fP)$ and $\fp=\vf^{-1}(P)$.

Our assumptions imply that $\p\in\UFD(A)$, and so 
there is an $A_{\p}$-module isomorphism $\fa_{\p}\cong A_{\p}$.
It follows that we have a $B_P$-module isomorphism 
$\fa_{\p}\otimes_{A_{\p}}B_P\cong B_P$.
In particular, this $B_P$-module is reflexive,
so the map
$$\bidual{B_P}{\fa_{\p}\otimes_{A_{\p}}B_P}\colon
\fa_{\p}\otimes_{A_{\p}}B_P\to\pddual{\fa_{\p}\otimes_{A_{\p}}B_P}{B_P}$$
is an isomorphism. From the following commutative diagram
$$\xymatrix{
(\fa\otimes_AB)_P \ar[rr]^-{(\bidual{B}{\fa\otimes_AB})_P} \ar[d]_{\cong}
&&(\pddual{\fa\otimes_AB}{B})_P\ar[d]^{\cong} \\
\fa_{\p}\otimes_{A_{\p}}B_P\ar[rr]^-{\bidual{B_P}{\fa_{\p}\otimes_{A_{\p}}B_P}}_-{\cong}
&&\pddual{\fa_{\p}\otimes_{A_{\p}}B_P}{B_P}
}$$
we conclude that  $(\bidual{B}{\fa\otimes_AB})_P$ is 
also an isomorphism.
Hence, the induced map
$$(\bidual{B}{\fa\otimes_AB})_P\otimes_{B_P}C_{\fP}\colon
(\fa\otimes_AB)_P\otimes_{B_P}C_{\fP}\to[\pddual{\fa\otimes_AB}{B}]_P\otimes_{B_P}C_{\fP}$$
is an isomorphism as well. Therefore, the next commutative diagram
$$\xymatrix{
[(\fa\otimes_AB)\otimes_BC]_{\fP} \ar[rrr]^-{(\bidual{B}{\fa\otimes_AB}\otimes_BC)_{\fP}} \ar[d]_{\cong}
&&&[\pddual{\fa\otimes_AB}{B}\otimes_BC]_{\fP}\ar[d]^{\cong} \\
(\fa\otimes_AB)_P\otimes_{B_P}C_{\fP}
	\ar[rrr]^-{(\bidual{B}{\fa\otimes_AB})_P\otimes_{B_P}C_{\fP}}_-{\cong}
&&&[\pddual{\fa\otimes_AB}{B}]_P\otimes_{B_P}C_{\fP}
}$$
shows that the map $(\bidual{B}{\fa\otimes_AB}\otimes_BC)_{\fP}$ is an isomorphism.

Since $\Upsilon$ is an isomorphism, the same is true of $\Upsilon_{\fP}$.
It follows that the map
$$\delta_{\fP}=[(\bidual{B}{\fa\otimes_AB}\otimes_BC)\circ \Upsilon]_{\fP}
=(\bidual{B}{\fa\otimes_AB}\otimes_BC)_{\fP}\circ \Upsilon_{\fP}
$$
is an isomorphism as well. From this, we conclude that the double-dual
$$\pddual{\delta_{\fP}}{C_{\fP}}
\colon\bddual{(\fa\otimes_AC)_{\fP}}{C_{\fP}}\to
\bddual{(\pddual{\fa\otimes_AB}{B}\otimes_BC)_{\fP}}{C_{\fP}}
$$
is an isomorphism. Finally, the commutative diagram
$$\xymatrix{
[\pddual{\fa\otimes_AC}{C}]_{\fP} 
	\ar[rrr]^-{(\ddual{\delta}{C})_{\fP}}
	\ar[d]_{\cong} 
&&& [\bddual{\pddual{\fa\otimes_AB}{B}\otimes_BC}{C}]_{\fP}
	\ar[d]_{\cong} \\
\bddual{(\fa\otimes_AC)_{\fP}}{C_{\fP}}
	\ar[rrr]^-{\pddual{\delta_{\fP}}{C_{\fP}}}_-{\cong}
&&& \bddual{(\pddual{\fa\otimes_AB}{B}\otimes_BC)_{\fP}}{C_{\fP}}
}$$
shows that $(\ddual{\delta}{C})_{\fP}$ is an isomorphism, as desired.
\end{proof}

\begin{corollary} \label{cor0101}
Let $\vf\colon A \to B$ and $\psi\colon B\to C$ 
be  ring homomorphisms of  normal integral domains. 
Assume that $\vf$ has finite flat dimension and that
$\psi$ either has finite flat dimension or is injective and integral.
Then each of the maps $\vf$, $\psi$ and $\psi\circ\vf$ satisfies
the hypotheses of Theorem~\ref{prop0101},
and $\Cl(\psi \circ\vf)=\Cl(\psi) \circ\Cl(\vf)$.
\end{corollary}

\begin{proof}
Our assumptions imply that $\vf$ and $\psi$ both satisfy
the hypotheses of Theorem~\ref{prop0101}; see
Theorem~\ref{prop0102} and Remark~\ref{rmk0102}.
In light of Theorem~\ref{thm0101}, it suffices to show that 
the composition $\psi\circ\vf$ satisfies
the hypotheses of Theorem~\ref{prop0101}.

If $\vf$ and $\psi$ both have finite flat dimension, then the composition
$\psi\circ\vf$ also has finite flat dimension and therefore satisfies
the hypotheses of Theorem~\ref{prop0101}. Assume now that 
$\psi$  is injective and integral. For each prime ideal
\newcommand{\fP}{\mathfrak{P}}
$\fP\in\Spec(C)$ of height at most one, 
the Cohen-Seidenberg Theorems imply that
the prime $P=\psi^{-1}(\fP)\in\Spec(B)$ has
$\operatorname{ht}(P)=\operatorname{ht}(\fP)\leq 1$.
Since $\vf$ has finite flat dimension, the prime
$\p=\vf^{-1}(P)=(\psi\circ\vf)^{-1}(\fP)\in\Spec(A)$
is in $\UFD(A)$, as desired.
\end{proof}

\section{Kernels of Homomorphisms Induced by Sequences of Ideals} \label{sec02}

The motivating principle  for the work in
this section is the idea that the pathological behavior of induced maps on divisor 
class groups, at least concerning injectivity, lies near the ``top'' of the maximal ideal
of a local ring.  
We begin by specifying some assumptions  for  the section.

\begin{assumption} \label{assn0201}
Throughout this section $(A,\m)$ is a local normal integral domain.
Fix  a sequence $I_1,I_2,\ldots$ of nonzero prime ideals of $A$ such that
$\lim_{n \to \infty} I_n = 0$ in the $\m$-adic topology. In other words,
for each integer $i\geq 1$, there is an integer $n_i\geq 1$ such that,
for all $n\geq n_i$ we have $I_n\subseteq\m^i$. 
We set
$P=\prod_n(A/I_n)$ and $S=\coprod_n(A/I_n)$, and we let
$\iota\colon A\to P$ denote the natural $A$-module homomorphism given by 
$\iota(a)=(a+I_1,a+I_2,\ldots)$.
Let $\pi\colon P\to P/S$ be the natural surjection.
\end{assumption}

\begin{lemma} \label{dprodsplits} 
With notation as in Assumption~\ref{assn0201}, if $A$ is complete, then the $A$-module homomorphism $\iota \colon A \to P$ is a split injection.
\end{lemma}

\begin{proof}
The proof is essentially the same as in \cite[Lem.\ 2.6]{Sp}, but with the principal 
ideals replaced by the $I_n$.  
\end{proof}

\begin{proposition}  \label{AdsumP/S} 
With notation as in Assumption~\ref{assn0201},
if $A$ is complete, then the $A$-module homomorphism $\pi\circ\iota\colon A\to P/S$ is a split injection.
\end{proposition}

\begin{proof} 
First, we show that $\pi \circ\iota$ is an injection. If $a\in\ker(\pi \circ\iota)$, then 
$$(a+I_1,a+I_2,\ldots)=\iota(a)\in S$$
so there is an integer $n_0\geq 1$ such that, for all $n\geq n_0$
we have $a+I_n=I_n$. In other words, we have
$a\in\cap_{n=n_0}^{\infty}I_n=0$.

Lemma~\ref{dprodsplits} yields an $A$-module homomorphism 
$\eta\colon P \to A$ such that $\eta\circ\iota=\id_A$. We claim that $\eta(S)=0$.
To show this, we fix an element $s\in S$.
Write $s=(s_1+I_1,s_2+I_2,\ldots)$ for some elements $s_i\in A$.
The condition $s\in S$ implies that there is an integer $N\geq 1$ such that,
for all $n\geq N$, we have $s_n+I_n=I_n$. Since $A$ is an integral domain
and each ideal $I_n$ is nonzero, there is a nonzero element $y\in I_1I_2\cdots I_N$.
It follows that $y  \eta(s) = \eta(ys) = 0$. Since $A$ is an integral domain and $y\neq 0$,
this implies $\eta(s)=0$, as claimed.

Hence, the map $\overline{\eta}\colon P/S \to A$ given by 
$\ol\eta\left(\ol p\right)=\eta (p)$ is a well-defined $A$-module homomorphism.
From the equality $\eta \circ\iota=\id_A$, it follows readily that
$\ol\eta \circ(\pi\circ\iota)=\id_A$, so that $\ol\eta$ is the desired splitting.
\end{proof}

The proof of Theorem~\ref{thm0201} relies on the notion of purity, which we discuss next.

\begin{definition} \label{defn0201}
A short exact sequence $\mathcal S \colon 0 \to M_1 \to M_2 \to M_3 \to 0$ of $A$-module
homomorphisms is 
\emph{pure exact} if, for each $A$-module $L$, the sequence $\mathcal S \otimes_A L$ is exact.  
(When $\mathcal S$ is 
pure exact, we also say that $M_1$ is a \emph{pure submodule} of $M_2$.)  
\end{definition}

\begin{remark} \label{rmk0201}
Let $\mathcal S$ 
be a short exact sequence of $A$-module homomorphisms. The sequence $\mathcal S$ is 
pure exact if and only if, 
for each finitely-presented $A$-module $L$, the 
sequence $\Hom_A(L,\mathcal S)$ is exact; 
see Fieldhouse~\cite[Cor.\ 7.1]{Fi}.  
As the $\m$-adic topology on $A$ is Hausdorff, 
the sequence $\mathcal S$ is pure exact if and only if, for each $A$-module
$L$ of finite length, the sequence $\mathcal S\otimes_AL$ is exact;
see Griffith~\cite[Cor.\ 3.2]{G2}.  
\end{remark}

\begin{lemma} \label{SPpure} 
With notation as in Assumption~\ref{assn0201},
the following exact sequence 
is pure exact: $0 \to S \to P \xra{\pi} P/S \to 0$.
\end{lemma}

\begin{proof}  
Let $\epsilon\colon S\to P$ be the natural inclusion.
Because of Remark~\ref{rmk0201}, we need only
show that the map $\epsilon\otimes_AL\colon S\otimes_AL\to P\otimes_AL$
is injective for each $A$-module $L$ of finite length. We have two sequences of natural isomorphisms:
\begin{align*}
S \otimes_A L
&\cong\biggl[\coprod_n(A/I_n)\biggr]\otimes_A L \cong \coprod_n (L/I_nL) \\
P \otimes_A L 
&\cong \biggl[\prod_n(A/I_n) \biggr]\otimes_A L \cong\prod_n (L/I_nL).
\end{align*}  
The final isomorphism uses the fact that $L$ is finitely generated;
see, e.g., \cite[Thm.\ 3.2.22]{EJ}.  
These isomorphisms fit into the following commutative diagram
$$\xymatrix{
S \otimes_A L \ar[r]^-{\epsilon \otimes_A L} \ar[d]_{\cong}
&P \otimes_A L\ar[d]^{\cong} \\
\coprod_n (L/I_nL)\ar[r]^-{\epsilon'}
&\prod_n (L/I_nL)
}$$
where $\epsilon'$ is the natural inclusion.
Since $\epsilon'$ is injective, the diagram shows that $\epsilon \otimes_A L$
is also injective, as desired.
\end{proof}

\begin{proposition} \label{P/S} 
With notation as in Assumption~\ref{assn0201},
if $M$ is a finitely generated $A$-module and $\Hom_A(M, A/I_n) \cong A/I_n$ for all 
$n$, then $\Hom_A(M, P/S) \cong P/S$.
\end{proposition}

\begin{proof}  
We have two sequences of isomorphisms
\begin{align*}
\Hom_A(M, P) 
&\cong \Hom_A\biggl(M,\prod_n(A/I_n)\biggr)\cong \prod_n\Hom_A(M,A/I_n)\cong\prod_n(A/I_n)
	\cong P \\
\Hom_A(M, S) 
&\cong \Hom_A\biggl(M,\coprod_n(A/I_n)\biggr)\stackrel{(\ast)}{\cong} \coprod_n\Hom_A(M,A/I_n)\cong
	\coprod_n(A/I_n)\cong S.
\end{align*}
The isomorphism $(\ast)$ uses the fact that $M$ is finitely generated.
Lemma~\ref{SPpure} shows that the sequence 
$0 \to S \to P \to P/S \to 0$ is pure exact, 
so Remark~\ref{rmk0201} implies that the top row of the following commutative diagram
is exact:
$$\xymatrix{
& 0 \ar[r] & \Hom_A(M, S) \ar[r] \ar[d]^{\cong} & \Hom_A(M, P) \ar[r] \ar[d]^{\cong}  
	&  \Hom_A(M, P/S) \ar[r]  & 0 \\
& 0 \ar[r] & S \ar[r] & P \ar[r] &P/S \ar[r] & 0.}$$
The vertical isomorphisms are the ones from the beginning of the proof.
Since both rows of this diagram are exact, a standard argument shows that
there is an induced isomorphism
$\Hom_A(M, P/S)\xra\cong P/S$, as desired. 
\end{proof}

We require two more key notions for the proof of Theorem \ref{thm0201}.

\begin{definition} \label{defn0202}
Let $M$ be an $A$-module and $N\subseteq M$ an $A$-submodule.
\begin{enumerate}[\quad(a)]
\item
We let $\ol N$ denote the $\m$-adic closure in $M$, that is,
$$\ol N=\{x \in M \mid \text{for each integer $r\geq 1$ we have $(x+\m^rM) \cap N \neq 
\emptyset$} \}.$$
\item
The $A$-module $M$ is \emph{$\m$-divisible} if $\m M = M$.  
\end{enumerate}
\end{definition}

\begin{lemma}  \label{mdiv} 
With notation as in Assumption~\ref{assn0201},
let $M$ be an $A$-module.
If  $N$ is a pure $A$-submodule of $M$, then $\overline{N}/N$ is the unique maximal 
element of the set of $\m$-divisible submodules of $M/N$, ordered by inclusion.
\end{lemma}

\begin{proof}  
We claim that the $\m$-adic topology on $N$ coincides with the subspace topology on $N$
induced by the $\m$-adic topology on $M$. 
To prove this, we fix an integer $n\geq 1$ and show that
$\m^n M \cap N=\m^n N$. The containment $\m^n M \cap N\supseteq\m^n N$ is standard.
For the reverse containment, 
let $x \in \m^n M \cap N$. Write $x = \sum_i a_i w_i$ where $a_i \in \m^n$ and $w_i \in M$.  
Because $N$ is pure, there exist $y_i \in N$ such that $x = \sum a_i y_i \in \m^n N$;
see~\cite[Thm.\ 7.13]{M}. This establishes the claim.

To complete the proof, it suffices to show that
$\ol N/N=\cap_{n=1}^{\infty}\m^n(M/N)$ since it is straightforward
to show that the right-hand side of this equality is the unique maximal 
$\m$-divisible submodule of $M/N$.
To verify the containment $\ol N/N\subseteq\cap_{n=1}^{\infty}\m^n(M/N)$,
we fix an element $z+N \in \overline{N}/N$.
The condition $z\in\ol N$ implies that,
for each integer $n\geq 1$,
there exists an element 
$v_n \in (z+\m^n M) \cap N$. 
Since we have $v_n\in N$, we conclude that
$(z-v_n)+N=z+N$. The inclusion $z-v_n \in \m^n M$ then implies that
$$z+N=(z-v_n)+N\subseteq\m^n M+N$$
and so 
$$z+N\in(\m^n M+N)/N=\m^n(M/N)$$
as desired.

For the reverse containment, let $z+N\in\cap_{n=1}^{\infty}\m^n(M/N)$.
For each $n\geq 1$, 
there is an element $w_n \in \m^nM$ such that
$z+N = w_n + N$.  It follows that we have $z-w_n \in (z+ \m^n M) \cap N$.
In particular, we conclude that $(z+ \m^n M) \cap N\neq \emptyset$ for each $n\geq 1$,
and so $z \in \overline{N}$.  
\end{proof}

\begin{remark} \label{rmk0202}
With notation as in Assumption~\ref{assn0201},
consider the Cartesian product $\prod A=\prod_{n=1}^{\infty}A$.
Given a sequence $(a_1,a_2,\ldots)\in\prod A$, we have
$\lim_{n\to\infty}a_n=0$ in the $\m$-adic topology when,
for each $i\geq 1$,
there exists $n_i\geq 1$ such that, for all $n\geq n_i$ we have $a_n\in\m^i$.
It is routine to show that the $\m$-adic closure of the submodule
$\coprod A=\coprod_{n=1}^{\infty}A\subseteq\prod A$ is 
\begin{align*}
\textstyle\ol{\coprod A}
& \textstyle 
	=\{(a_1,a_2,\ldots)\in\prod A\mid\text{$\lim_{n\to\infty}a_n=0$ in the
	$\m$-adic topology}\}. 
\end{align*}
In particular, we have $\ol{\coprod A}\subsetneq\prod A$. 

Similarly, since we have assumed that $\lim_{n\to \infty}I_n=0$ in the
$\m$-adic topology, it is readily shown that
$$\ol S=\{(a_1+I_1,a_2+I_2,\ldots)\in P\mid\text{$\lim_{n\to\infty}a_n=0$ in the
	$\m$-adic topology}\}.$$ 
In other words, if $\tau\colon \prod A\to P$ is the natural surjection,
then $\tau^{-1}(\ol S)=\ol{\coprod A}$.
\end{remark}

\begin{proposition}  \label{fflat} 
With notation as in Assumption~\ref{assn0201},
the $A$-module $P/\overline{S}$ is  flat.
\end{proposition}

\begin{proof}  
Let $\tau\colon \prod A\to P$ be the natural surjection.
From Remark~\ref{rmk0202}, it can be seen that 
$(\prod A)/\ol{\coprod A}=(\prod A)/\tau^{-1}(\ol S)\cong P/\ol S$.
This $A$-module is flat by~\cite[Lem.\ 2.5]{Sp}.
\end{proof}

We are now ready for the main point of this section.

\begin{para} \label{proofB}
\textbf{Proof of Theorem~\ref{thm0201}.}
We first prove the result in the case where $A$ is complete.  
For each index $n$, set $A_n = A/I_n$.  
Each map $\Cl(\pi_n)\colon\Cl(A) \to \Cl(A_n)$ is a 
well-defined group homomorphism by Theorem~\ref{prop0102},
given  by the rule $\cl{\fa}
\mapsto \cl{\ddual{(\fa \otimes_A A_n)}{A_n}}$.  
Fix an element $\cl\fa\in\bigcap_{n=1}^\infty \ker(\Cl(\pi_n))$. To show that
$\cl\fa$ is trivial, 
it suffices to show that $\cl\fa^{-1}=\cl{\dual{\fa}{A}}$ is
trivial, where $\dual{\fa}{A} = \Hom_A(\fa, A)$.  That is, we need to show that
$\dual{\fa}{A}\cong A$.

For each index $n$, the class $\cl{\ddual{(\fa \otimes_A A_n)}{A_n}}$ is trivial
in $\Cl(A_n)$, and hence the class
$\cl{\ddual{(\fa \otimes_A A_n)}{A_n}}^{-1}
=\cl{\dddual{(\fa \otimes_A A_n)}{A_n}}$
is also trivial. This explains the first isomorphism in the next sequence
$$A_n
\cong\dddual{(\fa \otimes_A A_n)}{A_n}
\cong\dual{(\fa \otimes_A A_n)}{A_n}
=\Hom_{A_n}(\fa \otimes_A A_n,A_n)
\cong\Hom_A(\fa,A_n).$$
The second isomorphism follows from Remark~\ref{rmk0103}, the equality
is by definition, and the final isomorphism is from Hom-tensor adjointness.
Proposition~\ref{P/S} implies that $\Hom_A(\fa, P/S) \cong P/S$.   
By Proposition \ref{AdsumP/S} there is an $A$-module $T$ such that
$P/S\cong A\oplus T$, and
it follows that 
$$P/S\cong\Hom_A(\fa, P/S)\cong\Hom_A(\fa, A\oplus T)
\cong \Hom_A(\fa, A)\oplus \Hom_A(\fa, T)
=\dual\fa A\oplus K$$
where $K=\Hom_A(\fa, T)$.
The natural exact sequence
$0 \to \overline{S}/S \to P/S  \to  P/\overline{S} \to 0$
then has the form
$$0 \to \overline{S}/S \xra{f} \dual{\fa}{A} \oplus K \to  P/\overline{S} \to 0.$$
Lemma~\ref{mdiv} implies that the module $\ol S/S$ is $\m$-divisible.
By Nakayama's Lemma, the only $\m$-divisible submodule of $\dual{\fa}{A}$
is 0; hence,  the image of the composition
$\ol S/S\to \dual{\fa}{A}\oplus K\to \dual{\fa}{A}$ is zero.
It follows that $\im(f)\subseteq K$, and we conclude that
$$P/\overline{S} \cong \dual{\fa}{A} \oplus K/\im(f).$$
By Proposition \ref{fflat}, the $A$-module $P/\overline{S}$ is flat.  Since $\dual{\fa}{A}$ 
is a direct summand of $P/\overline{S}$,  it follows that $\dual{\fa}{A}$ is flat, that is,
free as an $A$-module.  In other words, the class $\cl{\dual{\fa}{A}}$ is trivial, as
desired. This completes the proof when $A$ is complete.

Next, we establish the general case of the result. Since $A$ is an  excellent local
normal integral domain, the $\m$-adic completion $\widehat{A}$ is also a
local normal integral domain
with maximal ideal $\m\widehat A$. Each quotient $A/I_n$ is also excellent,
so the ring
$\widehat{A/I_n}\cong\widehat{A}/I_n\widehat{A}$
is also a local normal integral domain. Since $\widehat{A}$ is flat over $A$,
we have $I_n\widehat{A}\cong I_n\otimes_A\widehat{A}$, so
$$\pd_{\widehat{A}}(I_n\widehat{A})=\pd_{\widehat{A}}(I_n \otimes_A\widehat{A})
=\pd_A(I_n)<\infty.$$
Also, the containment $I_n\subseteq\m^i$ implies that
$I_n\widehat{A}\subseteq\m^i\widehat A=(\m_{\widehat A})^i$,
so we have $\lim_{n\to\infty}I_n\widehat{A}=0$ in the $\m\widehat A$-adic topology.
It follows that the sequence $I_1\widehat{A},I_2\widehat{A},\ldots$ of ideals in $\widehat A$
satisfies the hypotheses of the result. 
Let $\tau_n\colon\widehat{A}\to\widehat{A}/I_n\widehat{A}$ be the
natural surjection.
Since $\widehat{A}$ is complete,
the first part of the proof implies that
$\bigcap_{n=1}^\infty\ker(\Cl(\tau_n))$ is trivial.

The next display contains the natural commutative diagram of local ring homomorphisms
and the resulting commutative diagram of divisor class group homomorphisms.
$$\xymatrix{
A \ar[r]^g \ar[d]_{\pi_n} & \widehat A \ar[d]^{\tau_n} \\
A/I_n\ar[r] &\widehat{A}/I_n\widehat{A} }
\qquad \qquad \qquad
\xymatrix{
\Cl(A) \ar@{^(->}[r]^{\Cl(g)} \ar[d]_{\Cl(\pi_n)} & \Cl(\widehat A) \ar[d]^{\Cl(\tau_n)} \\
\Cl(A/I_n)\ar@{^(->}[r] & \Cl(\widehat{A}/I_n\widehat{A})}$$
Note that the commutativity of the second diagram follows from Theorem~\ref{thm0101},
and the injectivity of the horizontal maps follows because the corresponding
ring homomorphisms are  flat and local. 
Fix an element $\cl\fa\in\bigcap_{n=1}^\infty \ker(\Cl(\pi_n))$. 
The commutativity of the second diagram implies that
$$\Cl(g)(\cl\fa)\in\bigcap_{n=1}^\infty \ker(\Cl(\tau_n)).$$
This intersection is trivial,
so the injectivity of $\Cl(g)$ implies that $\cl\fa$ is trivial. \qed
\end{para}

\begin{corollary} \label{cor01}
Let $(A, \m, k$) be an excellent normal local integral
domain, and let $I_1, I_2, \dots$ be a sequence of 
nonzero prime ideals in $A$ of finite projective dimension.
Assume that $\lim_{n \to \infty} I_n = 0$ in the $\m$-adic
topology and that $A/I_n$ is a normal integral domain for each $n$.
For $n=1,2,\ldots$ let $\pi_n\colon A\to A/I_n$ be the natural surjection.
\begin{enumerate}[\quad\rm(a)]
\item\label{cor01a}
For each nontrivial divisor class $\cl\fa\in\Cl(A)$,
there is an integer $N_{\cl\fa}\geq 1$ such that,
for all $n\geq N_{\cl\fa}$ we have $\cl\fa\notin\ker(\Cl(\pi_n))$.
\item\label{cor01b}
If $\Cl(A)$ is finite, then there is an integer $N\geq 1$ such that
$\Cl(\pi_n)$ is injective for all $n\geq N$.
\end{enumerate}
\end{corollary}

\begin{proof}
\eqref{cor01a}
Suppose that no such integer exists. 
It follows that, for each integer $N\geq 1$, there is an integer $n\geq N$
such that $\cl\fa\in\ker(\Cl(\pi_n))$.
In other words, there is a strictly increasing sequence of integers
$1\leq n_1< n_2<\cdots$ such that $\cl\fa\in\ker(\Cl(\pi_{n_i}))$
for each index $i\geq 1$. However, the sequence of ideals
$I_{n_1},I_{n_2},\ldots$ satisfies the hypotheses of Theorem~\ref{thm0201},
so we have
$$\cl\fa\in\bigcap_{i=1}^\infty\ker(\Cl(\pi_{n_i})),$$
where the intersection, and hence $\cl\fa$, is trivial.  This is a contradiction.

\eqref{cor01b}
Since $\Cl(A)$ is finite,  the quantity
$N=\max\{N_{\cl\fa}\mid0\neq\cl\fa\in\Cl(A)\}$
is a well-defined integer. 
Fix a nontrivial divisor class $\cl\fa\in\Cl(A)$.
By definition of the $N_{\cl\fa}$,
one concludes that, for all $n\geq N\geq N_{\cl\fa}$,
we have $\cl\fa\notin\ker(\Cl(\pi_n))$. It follows that
$\Cl(\pi_n)$ is injective for all $n\geq N$, as desired.
\end{proof}

\section{Examples} \label{sec03}

In this section, we show how to construct sequences of ideals satisfying the
hypotheses of Theorem~\ref{thm0201}. Since the case where the ideals $I_n$ are principal
is covered by~\cite[Thm.\ 3.1]{Sp}, we exhibit examples that are not principal. Moreover, the 
following result, in conjunction with Example~\ref{ex0203},
shows how to construct examples that are not generated by regular sequences.

\begin{proposition} \label{prop0201}
Let $(A,\m)$ be an excellent local integral domain.
Assume that, for $n=1,2,\ldots$ and for $j=1,\ldots,6$ 
there exist elements $f_{n,j}\in\m$ such that, 
for $n=1,2,\ldots$ we have the following:
\begin{enumerate}[\quad\rm(1)]
\item \label{prop0201a}
The sequence $f_{n,1},\ldots,f_{n,6}$ is $A$-regular and contained in $\m^n$; and
\item \label{prop0201c}
The quotient $A/(f_{n,1},\ldots,f_{n,6})A$ is a normal integral domain.
\end{enumerate}
For $n=1,2,\ldots$ let $I_n$ be the ideal generated by the $2\times 2$ minors of the
matrix
$$F_n=
\begin{pmatrix} f_{n,1}& f_{n,2}& f_{n,3} \\  f_{n,4} &  f_{n,5} &  f_{n,6}
\end{pmatrix}.$$
Then each ideal $I_n\subseteq A$ is  prime and has
finite projective dimension over $A$, and each quotient $A/I_n$ is a normal integral domain.
Each ideal $I_n$ has height at most 2 and is minimally  generated by 3 elements;
in particular, $I_n$ is not generated by an $A$-regular sequence.
Furthermore, we have $I_n\subseteq\m^{2n}$ for each $n$, and thus
$\lim_{n\to\infty}I_n=0$ in the $\m$-adic topology.
\end{proposition}

\begin{proof}
\newcommand{\bbz}{\mathbb{Z}}
\newcommand{\ii}{\Im}
Consider the following matrix of independent polynomial variables
$$X=\begin{pmatrix} X_{1}& X_{2}& X_{3} \\  X_{4} &  X_{5} &  X_{6}
\end{pmatrix}$$
and let $\ii_2(X)\subseteq \bbz[X]=\bbz[X_1,\ldots,X_6]$ denote the ideal generated by
the $2\times 2$ minors of $X$. 
Let $\psi_n\colon\bbz[X]\to A$ be the unique ring homomorphism such that
$X_i\mapsto f_{n,i}$.

From Bruns and Vetter~\cite[Cor.\ (2.8)]{bruns:dr}
we conclude that the quotient $\bbz[X]/\ii_2(X)$ is a perfect $\bbz[X]$-module.
Since it is also a free $\bbz$-module, it is faithfully flat over $\bbz$.
(In the terminology of~\cite[Ch.\ 3]{bruns:dr}, these two conditions imply that
$\bbz[X]/\ii_2(X)$ is a \emph{generically perfect} $\bbz[X]$-module.)
It follows from~\cite[Thm.\ (3.9)]{bruns:dr} that the $A$-module
\begin{equation}\label{prop0201d}
\bbz[X]/\ii_2(X)\otimes_{\bbz[X]}A\cong A/\ii_2(X)A=A/I_n
\end{equation} 
is perfect; here we are tensoring along the homomorphism $\psi_n$.
In particular, the quotient $A/I_n$ has finite projective dimension
over $A$, and hence so does $I_n$.

For every field $K$, the ring $\bbz[X]/\ii_2(X)\otimes_{\bbz}K\cong K[X]/\ii_2(X)K[X]$ 
is a normal integral domain by~\cite[Thm.\ (2.11)]{bruns:dr}.
Using the sequence~\eqref{prop0201d}, we conclude from~\cite[Prop.\ (3.13)]{bruns:dr}
that $A/I_n$ is a normal integral domain.

It is straightforward to show that the ideal $I_n$ is minimally generated by 
the three $2\times 2$ minors of the matrix $F_n$. For the height computation,
we first note that $\Ht(\ii_2(X))=2$; this follows from the next sequence
$$2=\operatorname{grade}(\ii_2(X))\leq\Ht(\ii_2(X))\leq 2.$$
The equality in this sequence is from~\cite[Thm.\ (2.5)]{bruns:dr};
the first inequality is standard, and the second inequality 
follows from~\cite[Thm.\ (2.1)]{bruns:dr}. Hence, the inequality
$\Ht(I_n)\leq2$
is a consequence of~\cite[Thm.\ (3.16)]{bruns:dr}.

The final conclusions of the proposition follow directly from the definition
of $I_n$ and the assumption $f_{n,1},\ldots,f_{n,6}\in\m^n$.
\end{proof}

To conclude this section, we  construct an excellent normal integral
domain $A$ with a sequence of elements satisfying the hypotheses of
Proposition~\ref{prop0201}. 
This is accomplished in
Example~\ref{ex0203}. 
For clarity,
we complete the construction over the course of the following three
examples. Note that each of the examples contains a sequence
of ideals satisfying the hypotheses of Theorem~\ref{thm0201}.
In Example~\ref{ex0201}, each ideal is principal.
In Example~\ref{ex0202}, the ideals are generated by regular sequences of length 2.
Finally, Example~\ref{ex0203} contains ideals that are generated by regular sequences of length 6;
Proposition~\ref{prop0201} then applies to this example to yield ideals that 
satisfy the hypotheses of Theorem~\ref{thm0201}
and are
not generated by regular sequences.

\begin{example} \label{ex0201}
Let $(R, \m_R, k)$ be an excellent local normal domain containing the field $\mathbb Q$.
Consider the polynomial ring $P_1 = R[X_1, X_2,X_3]$ and the maximal ideal
$\n_1=(\m_R,X_1, X_2,X_3)P_1$, and set $A_1=(P_1)_{\n_1}$ with maximal ideal $\m_{A_1}=\n_1 A_1$.
For $n=1,2,\ldots$ we 
consider $f_n=X_1^n+X_2^n+X_3^n$ and set $I_n=(f_n)A_1$.
Each ideal $I_n$ is generated by the regular element $f_n$ and therefore has
finite projective dimension over $A_1$. Also, we have $\lim_{n\to\infty}I_n=0$ in
the $\m_{A_1}$-adic topology because each $f_n\in\m^n_{A_1}$. 
To show that the sequence of ideals 
$I_1,I_2, \dots$ satisfies the hypotheses of Theorem~\ref{thm0201},
we need to show that each ring $A_1/I_n$ is a normal domain.

\newcommand{\kp}{\kappa(\p)}
To this end, we note that 
the composition of natural maps $R\to A_1\to  A_1/I_n$ is a flat local homomorphism
by the corollary to~\cite[Thm.\ 22.6]{M}. 
For each prime ideal $\p\in\Spec(R)$, we set $\kp=(R/\p)_{\p}\cong R_{\p}\otimes_RR/\p$.
There is an isomorphism
$$\kappa(\p)\otimes_RA_1/I_n
\cong\left(\frac{\kp[X_1,X_2,X_3]}{(X_1^n+X_2^n+X_3^n)}\right)_{(\m_R,X_1,X_2,X_3)}$$
which is a localization of the ring
$B_n=\kp[X_1,X_2,X_3]/(X_1^n+X_2^n+X_3^n)$. 
The field $\kp$ has characteristic zero, so the Jacobian criterion shows that
$B_n$  is an isolated singularity, and therefore, it satisfies $(R_1)$.
It follows that $\kappa(\p)\otimes_RA_1/I_n$ also satisfies $(R_1)$.
Also, the ring $B_n$ is a hypersurface, so $\kappa(\p)\otimes_RA_1/I_n$ satisfies $(S_2)$.

In summary, the fibre ring $\kappa(\p)\otimes_RA_1/I_n$ is a normal integral domain
for each prime ideal $\p\in\Spec(R)$.
Since the ring $R$ is also a normal integral domain, the same is true of $A_1/I_n$ by \cite[Thm. 23.9]{M}
\end{example}

\begin{example} \label{ex0202}
Continue with the assumptions and notation of Example~\ref{ex0201}.
Consider the polynomial ring
$P_2 = R[X_1, X_2,X_3,Y_1,Y_2,Y_3]=P_1[Y_1,Y_2,Y_3]$ and the maximal ideal
$\n_2=(\m_R,X_1, X_2,X_3,Y_1,Y_2,Y_3)P_2$, 
and set $A_2=(P_2)_{\n_2}$ with maximal ideal $\m_{A_2}=\n_2 A_2$.
Note that $A_2$ is obtained from $A_1$ by the same procedure used to construct
$A_1$ from $R$. In particular, the ring $A_2$ is a flat $A_1$-algebra.

For $n=1,2,\ldots$ we 
consider  $g_n=Y_1^n+Y_2^n+Y_3^n$,
and we set $J_{n}=(f_n,g_n)A_2$.
Since the polynomial $f_n$ is prime in $A_1$,
it is also prime in $A_2$. 
Since $f_n$ and $g_n$ are given in terms of
independent sets of variables, it follows that
$g_n$ represents a nonzero element in the quotient
$A_2/(f_n)$. 
This quotient is an integral domain, so 
$g_n$ is $A_2/(f_n)$-regular.
In summary, 
for $n=1,2,\ldots$ the sequence $f_n,g_n$ is $A_2$-regular,
so the ideal $J_{n}$ has finite projective dimension over $A_2$.
Furthermore, the quotient $A_2/J_{n}$ can also be constructed from $A_1/I_n$
by the same procedure used to construct $A_1/I_n$ from $R$. 
Hence, Example~\ref{ex0201} applied to the ring $A_1/I_n$
shows that each quotient $A_2/J_{n}$ is a normal integral domain.
Finally, we have $f_n,g_n\in\m_{A_2}^n$ and it follows that
$\lim_{n\to\infty}J_{n}=0$ in the $\m_{A_2}$-adic topology.
In conclusion, the ring $A_2$ and the ideals $J_{1},J_{2},\ldots$
satisfy the hypotheses of Theorem~\ref{thm0201}.
\end{example}

\begin{example} \label{ex0203}
Let $(R, \m_R, k)$ be an excellent local normal integral domain with  $\operatorname{char}(k)=0$.
Let $\mathbf{X}=(X_{i,j})$ be a $3\times 6$ matrix of independent variables.
Consider the polynomial ring $P = R[\mathbf{X}]$ and the maximal ideal
$\n=(\m_R,\mathbf{X})P$, and set $A=P_{\n}$ with maximal ideal $\m=\n A$.
For $n=1,2,\ldots$ and $j=1,\ldots 6$ we 
consider the polynomial 
$f_{n,j}=X_{1,j}^n+X_{2,j}^n+X_{3,j}^n$.
An induction argument based on the contents of Example~\ref{ex0202}
shows that the polynomials $f_{n,j}$ satisfy the conditions
of Proposition~\ref{prop0201}.
\end{example}

\section*{Acknowledgments}

The authors are grateful to Phil Griffith for his thoughtful mentoring, and to the referee for helpful comments.


\providecommand{\bysame}{\leavevmode\hbox to3em{\hrulefill}\thinspace}
\providecommand{\MR}{\relax\ifhmode\unskip\space\fi MR }
\providecommand{\MRhref}[2]{%
  \href{http://www.ams.org/mathscinet-getitem?mr=#1}{#2}
}
\providecommand{\href}[2]{#2}

\end{document}